\documentclass{amsart}
\usepackage[all]{xy}
\usepackage{amssymb}

\newtheorem{thm}{Theorem}[section]
\newtheorem{prop}[thm]{Proposition} 
\newtheorem{lemma}[thm]{Lemma}
\newtheorem{cor}[thm]{Corollary}

\theoremstyle{definition} 
\newtheorem{definition}[thm]{Definition}

\theoremstyle{remark}
\newtheorem{remark}[thm]{Remark}

\newtheorem{example}[thm]{Example}

\newcommand{\from}{\colon}
\newcommand{\defequal}{:=}

\newcommand{\longto}{\longrightarrow}

\newcommand{\cM}{\mathcal{M}}
\newcommand{\cN}{\mathcal{N}}
\newcommand{\cD}{\mathcal{D}}
\newcommand{\cZ}{\mathcal{Z}}
\newcommand{\R}{\mathbb{R}}
\newcommand{\Z}{\mathbb{Z}}
\newcommand{\N}{\mathbb{N}}
\newcommand{\vect}{\mathfrak{X}}
\newcommand{\lie}{\mathcal{L}}
\newcommand{\pdiff}[2]{\frac{\partial #1}{\partial #2}}

\DeclareMathOperator{\im}{im}

\begin{document}
     
\title{Differential graded contact geometry and Jacobi structures}
\author{Rajan Amit Mehta}
\address{Department of Mathematics \& Statistics \\ 44 College Lane \\ Smith College \\ Northampton, MA 01063}
\email{rmehta@smith.edu}


\begin{abstract}
We study contact structures on nonnegatively-graded manifolds equipped with homological contact vector fields. In the degree $1$ case, we show that there is a one-to-one correspondence between such structures (with fixed contact form) and Jacobi manifolds. This correspondence allows us to reinterpret the Poissonization procedure, taking Jacobi manifolds to Poisson manifolds, as a supergeometric version of symplectization. 
\end{abstract}

\maketitle

\section{Introduction}
A manifold whose algebra of functions is equipped with a \emph{local Lie algebra} structure, in the sense of Kirillov \cite{kirillov}, is called a \emph{Jacobi manifold}. Equivalently, a Jacobi manifold is a manifold equipped with a bivector field $\Lambda$ and a vector field $R$, satisfying the equations \eqref{eqn:jacobi}. The definition of Jacobi manifolds in these terms is due to Lichnerowicz \cite{lich:cr, lich:jacobi}, who viewed it as a ``contravariant generalization of the notion of contact manifold.'' Since Poisson manifolds form the contravariant generalization of the notion of symplectic manifolds, Lichnerowicz' claim may be concisely described by the following analogy:
\begin{equation}\label{analogy}
\mbox{Jacobi : contact :: Poisson : symplectic}
\end{equation}
The following known results provide ways of formalizing this analogy:
\begin{itemize}
	\item There is a \emph{Poissonization} process, taking Jacobi manifolds to Poisson manifolds \cite{lich:jacobi}. This parallels the symplectization process that takes contact manifolds (with contact $1$-form) to symplectic manifolds.
	\item Jacobi manifolds ``integrate'' to \emph{contact groupoids} \cite{ker-sou, cra-zhu}. This parallels the fact that Poisson manifolds integrate to symplectic groupoids \cite{cdw}.
\end{itemize}
Note that these two approaches are ``orthogonal'' to each other, in terms of the analogy \eqref{analogy}. 

Although the latter approach is very interesting and potentially useful in many ways, it also has complications. In correspondence with the right side of analogy \eqref{analogy}, not every Jacobi manifold is integrable. For those that are integrable, one may not have a simpler description of the integration than as a quotient of an infinite-dimensional path space.

The purpose of this paper is to describe another way of connecting Jacobi and contact structures, allowing us to formalize the analogy \eqref{analogy} without the difficulties of integration. Namely, we show that there is a one-to-one correspondence between Jacobi manifolds and degree $1$ contact $\N Q$-manifolds with fixed contact form. This result parallels the well-known correspondence between Poisson manifolds and degree $1$ symplectic $\N Q$-manifolds. We furthermore show that, in this ``supergeometric'' point of view, Poissonization is \emph{the same thing} as symplectization in the $\N Q$ category. In other words, the following diagram commutes:
\begin{equation}\label{diag:poissonization}
	\xymatrix{ \fbox{Jacobi manifolds} \ar@{<->}[r] \ar_{\mbox{Poissonization}}[d] & \fbox{Deg.\ 1 contact $\N Q$-manifolds} \ar^{\mbox{Symplectization}}[d] \\
\fbox{Poisson manifolds} \ar@{<->}[r] & \fbox{Deg.\ 1 symplectic $\N Q$-manifolds}}
\end{equation}

The correspondence between Poisson manifolds and degree $1$ symplectic $\N Q$-manifolds has led to interesting results relating to Poisson reduction \cite{cz:reduction, mehta:homotopypoisson}. It also clarifies the relation between integration and quantization \cite{catt-feld:quantsympl}, via the AKSZ formalism \cite{aksz}. We believe that the correspondence between Jacobi manifolds and degree $1$ contact $\N Q$-manifolds should lead to analogous results. We plan to explore these ideas elsewhere.

Although the emphasis of this paper is on the degree $1$ case, we develop much of the general theory of contact $\N Q$-manifolds in arbitary degree. We remark that the degree $2$ case should provide a natural generalization of Courant algebroids, together with a ``Courantization'' process. This approach may be useful in studying Jacobi-Dirac and generalized contact structures \cite{wade:conformal, ponte-wade, poon-wade}.

The existence of a correspondence between Jacobi manifolds and degree $1$ contact $\N Q$-manifolds was known by \v{S}evera, who mentioned it in a footnote of \cite{sev}, but did not provide any details. More recently, Antunes and Laurent-Gengoux \cite{ant-lau} studied Jacobi bialgebroid structures from the supergeometric point of view. There are certainly relations between their results and ours, but neither is a special case of the other. Additionally, contact structures on supermanifolds were considered by Bruce \cite{bruce:contact, bruce:quasiq}. His papers played a role in inspiring the author to consider contact $\N Q$-manifolds. 

Shortly after a preprint version of this paper appeared on the arXiv, a preprint by Grabowski \cite{grabowski:gradedcontact} appeared, covering similar material, but with emphasis on the degree $2$ case. His work partially fulfills the above suggestion of a natural generalization of Courant algebroids. However, we should emphasize that there is an important distinction between his approach and ours. He identifies a graded contact manifold with its symplectization (which is larger but carries an $\R^\times$-action), and he develops the theory completely in terms of this identification. On the other hand, we show (Corollary \ref{cor:correspondence}) that a contact $\N$-manifold of degree $n>0$ (with fixed contact form) can be associated to a symplectic $\N$-manifold that is one dimension \emph{smaller} in degree $n$, and our results are stated in terms of this correspondence. The processes of reduction (by the $\R^\times$-action) and symplectization will allow one to translate between Grabowski's framework and ours.

\subsection{Conventions}
There are now many good introductions to the theory of $\Z$- and $\N$-graded manifolds, including  \cite{mehta:thesis,voronov:graded,roytenberg:graded,catt-schatz:super} (although, in contrast to \cite{voronov:graded}, we adopt a definition for which a function's parity agrees with its weight or degree). Roytenberg's paper \cite{roytenberg:graded} is particularly relevant, since it contains the details of the correspondence between Poisson manifolds and degree $1$ symplectic $\N Q$-manifolds, which plays an important role in motivating this paper. We will freely use his results on symplectic $\N$-manifolds without making explicit reference.

There are many possible sign conventions for the calculus of differential forms on a graded manifold. We will use the conventions of \cite{mehta:thesis, mehta:qalg}, where the algebra $\Omega(\cM)$ of differential forms on a graded manifold $\cM$ consists, by definition, of polynomial functions on $T[1]\cM$. Thus, the algebra $\Omega(\cM)$ is graded-commutative with respect to the total grading (i.e.\ the sum of the ``form'' grading and the internal ``manifold'' grading). When we say that a $p$-form is of degree $k$, we mean that the manifold grading is $k$.

With this choice of sign convention, the Cartan commutation relations include the following identities for any homogeneous vector fields $X,Y$ on $\cM$:
\begin{align*}
\lie_X &= [\iota_X, d] = \iota_X d + (-1)^{|X|}d \iota_X, \\
\iota_{[X,Y]} &= [\lie_X, \iota_Y] = \lie_X \iota_Y - (-1)^{|X|(|Y|-1)} \iota_Y \lie_X, \\
\iota_X \iota_Y &= (-1)^{(|X|-1)(|Y|-1)} \iota_Y \iota_X.
\end{align*}

On graded symplectic manifolds, we take Hamiltonian vector fields to be defined by the equation $df = (-1)^{|X|-1} \iota_X \omega$. Note that, if the degree of the symplectic form $\omega$ is $n$, then $|X| = |f| - n$. Poisson brackets are given by $\{f, g\} = X(g) = (-1)^{|Y| - 1}\iota_X \iota_Y \omega$, where $X$ and $Y$ are the Hamiltonian vector fields assocated to $f$ and $g$, respectively. The reader may verify that this convention gives the correct skew-commutativity rule for a degree $-n$ Lie bracket.

\subsection*{Acknowledgements} The author would like to thank Beno\^it Jubin for carefully reading a draft of the paper and for many interesting discussions on related topics. He would also like to thank the anonymous referees for suggestions that greatly improved both the mathematical content and the readability of the paper.

\section{Contact $N$-manifolds}\label{sec:contact}
In this section, we give the definition and some basic properties of degree $n$ contact $\N$-manifolds. Most of the results are straightforward extensions of well-known results from ordinary contact geometry. There are two features that are unique to the graded case. The first is the appearance of the Euler vector field, which automatically preserves the contact structure. The second is the fact (see Theorem \ref{thm:split}) that, when $n>0$, a degree $n$ contact $\N$-manifold with contact form naturally splits as the product of $\R[n]$ and a degree $n$ symplectic $\N$-manifold. This splitting gives a one-to-one correspondence between contact $\N$-manifolds with fixed contact form and symplectic $\N$-manifolds of degree $n>0$. At the end of the section, we discuss the noncoorientable case.

\subsection{Definition}
Let $\cM$ be an $\N$-graded manifold (or $\N$-manifold, for short), and let $\alpha$ be a nowhere-vanishing $1$-form of degree $n$ on $\cM$. The assignment $X \mapsto \iota_X \alpha$ is (left) $C^\infty(\cM)$-linear and so defines a degree $-n$ bundle map
\begin{equation*}
 \iota_\cdot \alpha \from T\cM \to \cM \times \R.
 \end{equation*}
The kernel $\cD \defequal \ker \iota_\cdot \alpha$ is a distribution of corank $1$ concentrated in degree $n$, so $L \defequal T\cM / \cD$ is a line bundle concentrated in degree $n$. Thus we have the  dual (up to grading shift) short exact sequences
\begin{align}
\cD &\longto T\cM \longto L, \label{eqn:sequence1} \\
\langle \alpha \rangle &\longto T^*\cM \longto \cD^*. \label{eqn:sequence2}
\end{align}
The assignment $X \mapsto \iota_X d \alpha$ is also $C^\infty(\cM)$-linear and so defines a degree $-n$ bundle map  
\begin{equation*}
(d\alpha)^\flat: T\cM \to T^*\cM.
\end{equation*}
We say that $\alpha$ is a \emph{contact $1$-form} if $(d\alpha)^\flat$ induces a bijection from $\cD$ to $\cD^*$, giving a nondegenerate pairing on $\cD$. 

The following statements are direct consequences of the definition.
\begin{lemma}\label{lemma:decomp}
     Let $\alpha$ be a contact $1$-form. Then
	\begin{enumerate}
		\item The cotangent bundle of $\cM$ splits as $T^* \cM = \im (d\alpha)^\flat \oplus \langle \alpha \rangle$, and
		\item the degree $n$ map 
		\begin{align*}
		\vect(\cM) &\to \Gamma(\im(d\alpha)^\flat) \oplus C^\infty(\cM) \\
		X &\mapsto (\iota_X d\alpha, \iota_X \alpha)
		\end{align*}
		 is an isomorphism of left $C^\infty(\cM)$-modules.
	\end{enumerate}
\end{lemma}

\begin{remark}
The nondegeneracy requirement imposes the same restrictions on the rank of $\cD$ that one sees on the dimensions of degree $n$ symplectic $\N$-manifolds; namely, the rank of $\cD$ in dimension $i$ equals the rank of $\cD$ in dimension $n-i$. Letting $\dim_i$ denote the dimension of $\cM$ in degree $i$, we deduce that $\dim_n \cM = \dim_0 \cM + 1$ and $\dim_i \cM = \dim_{n-i} \cM$ for $i > 0$. In particular, if an $\N$-manifold $\cM$ admits a degree $1$ contact $1$-form, then $\cM$ is necessarily concentrated in degrees $0$ and $1$, with $\dim_1 \cM = \dim_0 \cM + 1$.
\end{remark}

\begin{definition}
A \emph{contact $\N$-manifold of degree $n$} is an $\N$-manifold $\cM$ equipped with a distribution $\cD \subseteq T\cM$ that is locally the kernel of a contact $1$-form. 
\end{definition}
In general, a contact $1$-form associated to $\cD$ may only exist locally and is only well-defined up to multiplication by a nonvanishing degree $0$ function. If a contact $1$-form exists globally, then the contact structure $\cD$ is called \emph{coorientable}. From the short exact sequences \eqref{eqn:sequence1}--\eqref{eqn:sequence2}, we can see that $\cD$ is coorientable if and only if the degree $n$ line bundle $L$ is trivializable, and a choice of (local) contact $1$-form is equivalent to a choice of (local) trivialization of $L$.

\subsection{Contact vector fields}\label{sec:contactvf}
Let $(\cM, \cD)$ be a degree $n$ contact $\N$-manifold. A vector field $X \in \vect(\cM)$ is \emph{contact} if $[X,\Gamma(\cD)] \subseteq \Gamma(\cD)$. In terms of a contact $1$-form $\alpha$, contact vector fields are characterized by the property that there exists an $f \in C^\infty(\cM)$ such that
\begin{equation}\label{eqn:contactvf}
\lie_X \alpha = (-1)^{|X|} f \alpha.
\end{equation}
The sign in \eqref{eqn:contactvf} is only there to simplify the signs in later formulae.

In the remainder of \S\ref{sec:contactvf}, we will assume that $(\cM,\cD)$ is coorientable (or that we are working locally), and that we have fixed a choice of contact $1$-form $\alpha$.

We will now describe the contact analogue of Hamiltonian vector fields.  Let $h$ be a (homogeneous) function on $\cM$. Then, by Lemma \ref{lemma:decomp}, we may uniquely write 
\begin{equation}\label{eqn:dho}
dh = \beta + f \alpha,  
\end{equation}
where $\beta \in \im(d\alpha)^\flat$ and $f \in C^\infty(\cM)$. Again by Lemma \ref{lemma:decomp}, there exists a unique vector field $X_h$ such that
\begin{align}
	\iota_{X_h} d\alpha &= (-1)^{|h|-n+1} \beta, & \iota_{X_h} \alpha &= h. \label{eqn:ham}
\end{align}
In this case, we have that $|X_h| = |h| - n$. Then 
\begin{equation*}
	\begin{split}
		\lie_{X_h} \alpha &= \iota_{X_h} d \alpha + (-1)^{|X_h|}d \iota_X \alpha \\
     &= (-1)^{|X_h|+1}\beta + (-1)^{|X_h|}dh\\
     &= (-1)^{|X_h|} f \alpha,
	\end{split}
\end{equation*}
so $X_h$ is contact.

The process of taking functions to contact vector fields is invertible; given a contact vector field $X$, one can recover a function $h$ via \eqref{eqn:ham}, and the contact vector field associated to $h$ is again $X$. In summary, we have the following:
\begin{prop}
	\label{prop:contactvf}
There is a one-to-one correspondence between functions and contact vector fields on $\cM$. Functions of degree $k$ correspond to contact vector fields of degree $k - n$.
\end{prop}
\begin{example}
	The \emph{Reeb vector field} $\rho$ is the degree $-n$ vector field defined by the equations $\iota_\rho d\alpha = 0$ and $\iota_\rho \alpha = 1$. Under the correspondence of Proposition \ref{prop:contactvf}, the Reeb vector field correponds to the constant function $1$.
\end{example}

The Reeb vector field allows us to explicitly perform the decomposition \eqref{eqn:dho}. Note that, if $\beta \in \im(d\alpha)^\flat$, then $\iota_\rho \beta = 0$. Thus, for any $h \in C^\infty(\cM)$, with $\beta$ and $f$ given by \eqref{eqn:dho}, we have that
\begin{equation}\label{eqn:f}
 \rho(h) = \iota_\rho dh = \iota_\rho f\alpha = (-1)^{(n-1)|f|} f.
\end{equation}
This allows us to solve for $\beta$:
\begin{equation}\label{eqn:beta}
 \beta = dh - (-1)^{(n-1)|h|} \rho(h) \alpha.
\end{equation}

\begin{example}
	The \emph{Euler vector field} $\varepsilon$ is the degree $0$ vector field given by $\varepsilon(f) = |f|f$ for any homogeneous function $f$. Since $\alpha$ is of degree $n$, we have that $\lie_\varepsilon \alpha = n\alpha$, so the Euler vector field is contact. Let $\theta \defequal \iota_\varepsilon \alpha$ be called the \emph{Euler function} of $(\cM, \alpha)$. The degree of $\theta$ is $n$.
\end{example}
\begin{lemma}\label{lemma:reebeuler}
	The Reeb vector field and the Euler function satisfy the equation $\rho(\theta) = n$.
\end{lemma}
\begin{proof}
	Using the definition of $\theta$, we have that
\begin{equation*}
	\begin{split}
		\rho(\theta) &= \lie_\rho \iota_\varepsilon \alpha \\
	       &= \iota_\rho \lie_\varepsilon \alpha + \iota_\rho \iota_\varepsilon d \alpha.
	\end{split}
\end{equation*}
The latter term vanishes because $\iota_\rho d\alpha = 0$, and the first term is $n \iota_\rho \alpha = n$.
\end{proof}

\begin{remark}\label{rmk:rhotheta}
It should be emphasized that the correspondence of Proposition 2.4 depends on the choice of $\alpha$. In particular, the Euler function $\theta$ and the Reeb vector field $\rho$ both depend on $\alpha$. If $\alpha' = f\alpha$ for a nonvanishing degree $0$ function $f$, then the corresponding Reeb vector field $\rho'$ is equal to the contact vector field $X_{1/f}$. The corresponding Euler function $\theta'$ is $f\theta$.
\end{remark}

\subsection{The structure of contact $\N$-manifolds}\label{sec:structure}
Let $(\cM,\cD)$ be a degree $n$ contact $\N$-manifold where $n > 0$. We first consider the coorientable case. In this case, fix a contact $1$-form $\alpha$, let $\rho$ be the associated Reeb vector field, and let $\theta$ be the associated Euler function.

Although odd vector fields aren't automatically integrable, we see that $[\rho,\rho]$ vanishes for any $n$, since there are no nontrivial degree $-2n$ vector fields on $\cM$. Thus $\rho$ is integrable, and since it is nonvanishing (it satisfies $\iota_\rho \alpha = 1$), it induces a free $\R[n]$-action on $\cM$. An important feature that distinguishes the $n>0$ case from ordinary contact geometry is that the Reeb orbits cannot contain any topological complexities, such as being dense or incomplete. The reason is simply that $\R[n]$ is not really a line; rather, it is a sheaf over a single point. 

Let $\lambda \defequal \alpha - \frac{1}{n} d\theta$ and $\omega \defequal d\lambda = d\alpha$.

\begin{thm}\label{thm:split}
Let $(\cM, \alpha)$ be an $\N$-manifold equipped with a degree $n$ contact $1$-form for $n>0$. Then $\cM$ is the total space of a principal $\R[n]$-bundle with a canonical trivialization. The $1$-form $\lambda$ is basic, and $\omega$ passes to a degree $n$ symplectic form on the quotient.
\end{thm}

\begin{proof}
Let $\cN$ be the quotient of $\cM$ by the action of the Reeb vector field (which is a free $\R[n]$ action). The graded algebra of functions on $\cN$ consists of those functions $f$ on $\cM$ for which $\rho(f) = 0$. On the other hand, the function $\theta$ determines a projection map $\cM \to \R[n]$ that trivializes the principal $\R[n]$-bundle $\cM \to \cN$.

Using Lemma \ref{lemma:reebeuler}, we see that $\iota_\rho \lambda = 0$ and $\lie_\rho \lambda = 0$, so $\lambda$ is basic. Nondegeneracy of the push-forward of $\omega$ to $\cN$ follows from the fact that $\rho$ spans the characteristic distribution for the presymplectic form $d\alpha$.
\end{proof}

In Theorem \ref{thm:split}, the contact form on $\cM$ can be recovered from the symplectic form $\omega$ on $\cN$, since 
\begin{align}\label{eqn:alphaomega}
\lambda = \frac{1}{n}\iota_\varepsilon \omega	&& \mbox{and} && \alpha = \lambda + \frac{1}{n}d\theta.
\end{align}
Furthermore, given any degree $n$ symplectic $\N$-manifold $(\cN, \omega)$ for $n>0$, equations \eqref{eqn:alphaomega} define a degree $n$ contact form on $\cN \times \R[n]$. We therefore have the following result:

\begin{cor}\label{cor:correspondence}
When $n>0$, there is a one-to-one correspondence between degree $n$ symplectic $\N$-manifolds and degree $n$ contact $\N$-manifolds with fixed contact form.
\end{cor}

\begin{example}\label{example:deg1}
Since every degree $1$ symplectic $\N$-manifold is canonically symplectomorphic to $T^*[1]M$ for some manifold $M$, we have that every degree $1$ contact $\N$-manifold with fixed contact form is canonically of the form $T^*[1]M \times \R[1]$ with the standard contact form $\alpha = \lambda + d\theta$, where $\lambda$ is the Liouville $1$-form on $T^*[1]M$, and $\theta$ is the coordinate function on $\R[1]$. The Reeb vector field is $\rho = \pdiff{}{\theta}$.
\end{example}

In Theorem \ref{thm:split}, the bundle structure itself depends on $\rho$ (which, in turn, depends on $\alpha$), so it isn't possible to use Theorem \ref{thm:split} locally and patch the results to obtain a global result. However, we may take a different approach in the non-coorientable case. It is clear from the way $\theta$ transforms (see Remark \ref{rmk:rhotheta}) that the ideal generated by $\theta$ is independent of the choice of $\alpha$.  The resulting sheaf of ideals may be viewed as a canonical submanifold $\cZ$ of codimension $1$ concentrated in degree $n$. Any choice of local contact form on $\cM$ corresponds to a choice of local symplectic form on $\cZ$, since $\cZ$ is the zero section of the bundle in Theorem \ref{thm:split}. Under a change of contact form $\alpha' = f\alpha$, the corresponding symplectic form on $\cZ$ transforms as 
\begin{equation*}
\omega' = f \omega + \frac{1}{n}df \wedge (\iota_\varepsilon \omega). 
\end{equation*}

\section{Contact $\N Q$-manifolds}

Recall that a \emph{homological vector field} on a graded manifold $\cM$ is a degree $1$ vector field $Q$ such that $Q^2 = 0$. If $\cM$ has a contact structure, then we may consider vector fields that are both contact and homological.
\begin{definition}
	A \emph{degree $n$ contact $\N Q$-manifold} is a degree $n$ contact $\N$-manifold $\cM$, equipped with a vector field $Q$ that is contact and homological.
\end{definition}

\subsection{The case $n=1$}\label{sec:jacobi}
In this section, we describe the correspondence between degree $1$ contact $\N Q$-manifolds and Jacobi manifolds.

Recall from Example \ref{example:deg1} that every degree $1$ contact $\N$-manifold with fixed contact form is canonically of the form $\cM = T^*[1]M \times \R[1]$ for some ordinary manifold $M$. We remind the reader that functions on $T^*[1]M$ can be identified with multivector fields on $M$.

Let us first describe degree $1$ contact vector fields on $\cM$. By Proposition \ref{prop:contactvf}, every degree $1$ contact vector field arises from a degree $2$ function $h$ on $\cM$. Any such function is of the form
\begin{equation*}
	h = \Lambda + \theta R,
\end{equation*}
where $\Lambda$ is a bivector field and $R$ is a vector field on $M$. Following \eqref{eqn:dho}, \eqref{eqn:f}, and \eqref{eqn:beta}, we write
\begin{equation}\label{eqn:dh}
	dh = d\Lambda - \theta dR + R d\theta = d\Lambda - \theta dR - R\lambda + R\alpha,
\end{equation}
where $\beta \defequal d\Lambda - \theta dR - R\lambda$ is in $\im (d\alpha)^\flat$. The corresponding contact vector field $Q$ is defined by the equations \eqref{eqn:ham}, which in this case become
\begin{align}
	\iota_Q \alpha &= \Lambda + \theta R, \label{eqn:iqalpha} \\
     \iota_Q d\alpha &= d\Lambda - \theta dR - R\lambda. \label{eqn:iqdalpha}
\end{align}
The unique solution is
\begin{equation}\label{eqn:q}
	Q = X_\Lambda + \theta X_R - R\varepsilon - (\Lambda + \theta R) \pdiff{}{\theta},
\end{equation}
where $X_\Lambda$ and $X_R$ are the Hamiltonian vector fields on $T^*[1]M$ associated to $\Lambda$ and $R$, respectively. The Hamiltonian vector fields annihilate $\theta$ and act on multivector fields via the Schouten bracket.

The verification of \eqref{eqn:q} is a straightforward exercise, using the definition of Hamiltonian vector fields and \eqref{eqn:alphaomega}. Applying Proposition \ref{prop:contactvf}, we have the following result:
\begin{prop}
	Every degree $1$ contact vector field $Q$ on $\cM$ is of the form \eqref{eqn:q} for some $\Lambda \in \vect^2(M)$ and $R \in \vect^1(M)$.
\end{prop}
Comparing \eqref{eqn:dh} with the construction of \S\ref{sec:contactvf}, we have that 
\begin{equation}\label{eqn:lieq}
\lie_Q \alpha = -R\alpha.	
\end{equation}
Next, we consider the conditions on $\Lambda$ and $R$ that arise from the requirement $Q^2=0$. 
\begin{prop}\label{prop:qsquared}
	Let $Q$ be a contact vector field of the form \eqref{eqn:q}. Then $Q^2 = 0$ if and only if
\begin{align}\label{eqn:jacobi}
	[\Lambda, \Lambda] = 2R\Lambda && \mbox{and} && [R,\Lambda] = 0.
\end{align}
\end{prop}
\begin{proof}
Contact vector fields are closed under the Lie bracket, so $Q^2 = \frac{1}{2}[Q,Q]$ is contact. By the correspondence of Proposition \ref{prop:contactvf}, we have that $Q^2 = 0$ if and only if $\iota_{[Q,Q]} \alpha = 2\iota_{Q^2} \alpha = 0$. Using \eqref{eqn:iqalpha}, \eqref{eqn:q}, and \eqref{eqn:lieq}, we then compute
\begin{equation*}
	\begin{split}
		\iota_{[Q,Q]} \alpha &= \lie_Q \iota_Q \alpha - \iota_Q \lie_Q \alpha \\
&= \lie_Q (\Lambda + \theta R) + \iota_Q (R\alpha) \\
&= [\Lambda, \Lambda] - 2R\Lambda + 2\theta[R,\Lambda],
	\end{split}
\end{equation*}
which vanishes if and only if the equations \eqref{eqn:jacobi} hold.
\end{proof}

The equations \eqref{eqn:jacobi} are exactly those that define a Jacobi structure on $M$. Thus, we have shown the following:

\begin{thm}\label{thm:jacobicontact}
	There is a one-to-one correspondence between Jacobi manifolds and degree $1$ contact $\N Q$-manifolds with fixed contact form.
\end{thm}

\begin{remark}
	It is well-known \cite{ker-sou,vaisman} that one can associate to a Jacobi manifold $M$ a Lie algebroid structure on $T^*M \times \R$. The search for a converse result, characterizing those Lie algebroid structures on $T^*M \times \R$ that arise from Jacobi structures, led Iglesias and Marrero \cite{iglesias-marrero} and Grabowski and Marmo \cite{grab-marmo} to define the notion of \emph{Jacobi bialgebroids}. 

We can interpret Theorem \ref{thm:jacobicontact} as giving a more direct answer to the same question. Namely, the Lie algebroid structures on $T^*M \times \R$ that arise from Jacobi structures are exactly those for which the Lie algebroid differential is a contact vector field.
\end{remark}

\begin{remark}
A key point that is implicit in Theorem \ref{thm:jacobicontact} is that, once a contact form is fixed, then the identification of $\cM$ with $T^*[1]M \times \R[1]$ is canonical. However, this identification does depend on the choice of contact form, and a different choice will lead to a different (but conformally equivalent) Jacobi structure. This gives a one-to-one correspondence between Jacobi structures up to conformal equivalence and coorientable degree $1$ contact $\N Q$-manifolds. Going farther, we may obtain a correspondence between local Lie algebras and degree $1$ contact $\N Q$-manifolds.
\end{remark}

\section{Symplectization}
Let $\cM$ be a degree $n$ contact $\N$-manifold with fixed contact form $\alpha$. On $\cM \times \R$, one defines a $2$-form $\tilde{\omega} = d(e^t \alpha) = e^t(dt \cdot \alpha + d\alpha)$. Since the coordinate $t$ on $\R$ is of degree $0$, we have that $\tilde{\omega}$ is of degree $n$. The assumptions on $\alpha$ imply that $\tilde{\omega}$ is nondegenerate, so $\cM \times \R$ is a degree $n$ symplectic manifold. The process taking $(\cM, \alpha)$ to $(\cM \times \R, \tilde{\omega})$ is called \emph{symplectization}.

The following lemmas describe the relationship between symplectization and the respective symmetries of the contact and symplectic structures.

\begin{lemma}\label{lemma:ham}
	Let $X \in \vect(\cM)$ be a contact vector field, and let $f$ be the corresponding function in \eqref{eqn:contactvf}.  Then $X - f \pdiff{}{t}$ is a Hamiltonian vector field on $\cM \times \R$, with Hamiltonian function $H_X \defequal e^t(\iota_X \alpha)$.
\end{lemma}
\begin{proof}
	On the one hand, we have that 
\begin{equation*}
d H_X = d(e^t \iota_X \alpha) = e^t(dt \cdot \iota_X \alpha + d \iota_x \alpha).	
\end{equation*}
On the other hand, we have that 
\begin{equation*}
		\iota_{X - f \pdiff{}{t}} \tilde{\omega} = e^t\left( (-1)^{|X|-1} dt \cdot \iota_X \alpha + \iota_X d\alpha - f\alpha \right).
\end{equation*}
The conclusion follows from \eqref{eqn:contactvf} and the identity $\lie_X = \iota_X d + (-1)^{|X|} d \iota_X$.
\end{proof}

\begin{lemma}\label{lemma:hamhom}
	Let $Q \in \vect(\cM)$ be a homological contact vector field, and let $\varphi$ be the degree $1$ function such that $\lie_Q \alpha = -\varphi \alpha$. Then the Hamiltonian vector field $Q - \varphi \pdiff{}{t}$ on $\cM \times \R$ is also homological.
\end{lemma}
\begin{proof}
	On the one hand, $(\lie_Q)^2 \alpha = \lie_Q (-\varphi \alpha) = -(Q (\varphi)) \alpha$. In the last step, we have used the fact that $\varphi^2 = 0$. On the other hand, $(\lie_Q)^2 \alpha = \lie_{Q^2} \alpha = 0$, since $Q$ is homological. It follows that $Q(\varphi) = 0$.

     Now, we can directly see that $(Q - \varphi \pdiff{}{t})^2 = Q^2 - Q(\varphi) \pdiff{}{t} = 0$.
\end{proof}
Together, Lemmas \ref{lemma:ham} and \ref{lemma:hamhom} give the following result.
\begin{thm}\label{thm:symplectization}
	The symplectization process takes contact $\N Q$-manifolds with fixed contact form to symplectic $\N Q$-manifolds.
\end{thm}

\subsection{Poissonization}

We now return to the case $n=1$, where $\cM = T^*[1]M \times \R[1]$, with the standard contact form $\lambda + d\theta$, where $\lambda$ is the Liouville $1$-form on $T^*[1]M$. The symplectization process gives $T^*[1]M \times \R[1] \times \R$, with the symplectic form $\tilde{\omega} = e^t (dt (d\theta + \lambda) + \omega)$, where $\omega$ is the canonical symplectic form on $T^*[1]M$. 

Given a Jacobi structure $(\Lambda, R)$ on $M$, we have a homological contact vector field $Q$, given by \eqref{eqn:q}. Lemmas \ref{lemma:ham} and \ref{lemma:hamhom} tell us that $Q$ induces a homological Hamiltonian vector field on the symplectization $T^*[1]M \times \R[1] \times \R$, with Hamiltonian function $H_Q = e^t(\iota_Q \alpha) = e^t (\Lambda + \theta R)$; here, we have used \eqref{eqn:iqalpha}.

In order to realize $H_Q$ as a bivector field on $M \times \R$, we need to transform $\tilde{\omega}$ into the canonical symplectic form $\omega + dt d\theta$, arising from the obvious identification of $T^*[1]M \times \R[1] \times \R$ with $T^*[1](M \times \R)$.

Consider the diffeomorphism $\xi$ of $T^*[1]M \times \R[1] \times \R$, given by
\begin{equation*}
	\xi^* f = \exp(t\varepsilon)(f) = e^{|f|t} f
\end{equation*}
for any homogeneous function $f$. Using the power series expansion of the exponential, we can see that, for any homogeneous differential form $\beta$,
\begin{equation*}
	\xi^* \beta = \exp(\lie_{t\varepsilon})(\beta) = e^{|\beta|t} \left(\beta + dt \iota_\varepsilon \beta \right).
\end{equation*}
In particular, $\xi^* \omega = e^t(\omega + dt \cdot \lambda)$, and $\xi^*(dt d\theta) = e^t dt d\theta$. Thus, $\xi^*(\omega + dtd\theta) = \tilde{\omega}$. In other words:
\begin{prop}
	The diffeomorphism $\xi$ relates the symplectic form $\tilde{\omega}$ with the canonical symplectic form on $T^*[1](M \times \R)$.
\end{prop}
Since $H_Q$ is of degree $2$, its push-forward by $\xi$ is
\begin{equation*}
	\xi_* H_Q = e^{-t}(\Lambda + \theta R),
\end{equation*}
which exactly correponds to the bivector field for the Poissonization of the Jacobi structure $(\Lambda, R)$ (since $\theta$ is the conjugate variable that plays the role of $\pdiff{}{t}$). Thus we have shown the following.
\begin{thm}
	The diagram \eqref{diag:poissonization} commutes.
\end{thm}

\bibliographystyle{amsalpha}
\bibliography{contactbib}

\end{document}